\documentclass[12pt]{amsart}
\usepackage[utf8]{inputenc}

\usepackage{subfiles}
\usepackage{hyperref}
\usepackage{cleveref}
\usepackage{amsmath}
\usepackage{multirow}
\usepackage{amssymb}
\usepackage{amsthm}
\usepackage{mathtools}
\usepackage{array}
\usepackage{thmtools}
\usepackage{thm-restate}
\usepackage[margin=1in]{geometry}
\usepackage{cleveref}
\usepackage{ytableau}
\usepackage{comment}
\usepackage{soul}
\usepackage{tikz}
\usetikzlibrary{patterns, matrix}
\tikzstyle{NE-lines}=[pattern=north east lines, pattern color=black!45]
\usetikzlibrary{positioning}
\usepackage{booktabs}

\usepackage{graphicx}
\usepackage{caption}
\usepackage{float}

\usepackage[backend=bibtex]{biblatex}
\renewbibmacro{in:}{}
\DeclareFieldFormat[article]{citetitle}{#1}
\DeclareFieldFormat[article]{title}{#1}

\numberwithin{equation}{section}

\newcommand{\machine}[9]{\begin{tikzpicture}[scale=0.55]
\draw[thick] (0,0) -- (1,0) -- (1,-2.5) -- (2,-2.5) -- (2,0) -- (3.5,0) -- (3.5,-2.5) -- (4.5,-2.5) -- (4.5, 0) -- (5.7, 0);
\node[fill = white, draw = white] at (5,.5) {#1};
\node[fill = white, draw = white] at (5,-0.8) {\tiny 1};
\node[fill = white, draw = white] at (5,-1.4) {\tiny3};
\node[fill = white, draw = white] at (5,-2) {\tiny2};
\node[fill = white, draw = white] at (5.5,-0.8) {\tiny3};
\node[fill = white, draw = white] at (5.5, -1.4) {\tiny2};
\node[fill = white, draw = white] at (5.5,-2) {\tiny1};
\node[fill = white, draw = white] at (2.4,-1.1) {\tiny2};
\node[fill = white, draw = white] at (2.4,-1.7) {\tiny1};
\node[fill = white, draw = white] at (4,-2) {#4};
\node[fill = white, draw = white] at (4,-1.3) {#3};
\node[fill = white, draw = white] at (4,-.6) {#2};
\node[fill = white, draw = white] at (2.75,.5) {#5};
\node[fill = white, draw = white] at (1.5,-2) {#8};
\node[fill = white, draw = white] at (1.5,-1.3) {#7};
\node[fill = white, draw = white] at (1.5,-.6) {#6};
\node[fill = white, draw = white] at (0.5,.5) {#9};
\node[fill = white, draw = white] at (6.75,-1.4) {$\rightarrow$};
\end{tikzpicture}}

\newcommand{\machineend}[9]{\begin{tikzpicture}[scale=0.55]
\draw[thick] (0,0) -- (1,0) -- (1,-2.5) -- (2,-2.5) -- (2,0) -- (3.5,0) -- (3.5,-2.5) -- (4.5,-2.5) -- (4.5, 0) -- (5.7, 0);
\node[fill = white, draw = white] at (5,.5) {#1};
\node[fill = white, draw = white] at (5,-0.8) {\tiny 1};
\node[fill = white, draw = white] at (5,-1.4) {\tiny3};
\node[fill = white, draw = white] at (5,-2) {\tiny2};
\node[fill = white, draw = white] at (5.5,-0.8) {\tiny3};
\node[fill = white, draw = white] at (5.5, -1.4) {\tiny2};
\node[fill = white, draw = white] at (5.5,-2) {\tiny1};
\node[fill = white, draw = white] at (2.4,-1.1) {\tiny2};
\node[fill = white, draw = white] at (2.4,-1.7) {\tiny1};
\node[fill = white, draw = white] at (4,-2) {#4};
\node[fill = white, draw = white] at (4,-1.3) {#3};
\node[fill = white, draw = white] at (4,-.6) {#2};
\node[fill = white, draw = white] at (2.75,.5) {#5};
\node[fill = white, draw = white] at (1.5,-2) {#8};
\node[fill = white, draw = white] at (1.5,-1.3) {#7};
\node[fill = white, draw = white] at (1.5,-.6) {#6};
\node[fill = white, draw = white] at (0.5,.5) {#9};
\end{tikzpicture}}

\theoremstyle{definition}
\newtheorem{theorem}{Theorem}[section]
\newtheorem*{theorem*}{Theorem}

\newtheorem*{example*}{Example}
\newtheorem{lemma}[theorem]{Lemma}
\newtheorem*{lemma*}{Lemma}
\newtheorem{corollary}[theorem]{Corollary}
\newtheorem*{corollary*}{Corollary}

\newtheorem*{definition*}{Definition}
\newtheorem{proposition}[theorem]{Proposition}
\newtheorem*{proposition*}{Proposition}

\newtheorem*{remark*}{Remark}
\newtheorem{conjecture}[theorem]{Conjecture}

\newcommand{\Av}{\mathrm{Av}}

\usepackage{ mathrsfs }

\newcommand{\indx}{\mathrm{ind}} %macro for ind
\newcommand{\smallx}{\mathrm{small}} %macro for small
\newcommand{\Act}{\mathrm{Ac}} %macro for Ac
\newcommand{\Sort}{\mathrm{Sort}} %macro for Sort

\title[On a conjecture on pattern-avoiding machines]{On a conjecture on pattern-avoiding machines}

\author{Christopher Bao}\address{\textsc{C. Bao}, The Davidson Academy, Reno, NV, 89507} \email{christopherbao@yahoo.com}

\author{Giulio Cerbai}\address{\textsc{G. Cerbai}, University of Iceland, Iceland} \email{giulio@hi.is}

\author{Yunseo Choi}\address{\textsc{Y. Choi}, Harvard University,
    Cambridge, MA, 02138} \email{ychoi@college.harvard.edu}

\author{Katelyn Gan}\address{\textsc{K. Gan}, Sage Hill School,
    Newport Beach, CA, 92657} \email{katelyngan77@gmail.com}

\author{Owen Zhang}\address{\textsc{O. Zhang}, Interlake High School,
    Bellevue, WA, 98008} \email{owenzhang128@gmail.com}

\begin{document}

\begin{abstract}
Let $s$ be West's stack-sorting map, and let $s_{T}$ be the generalized stack-sorting map, where instead of being required to increase, the stack avoids subpermutations that are order-isomorphic to any permutation in the set $T$. In 2020, Cerbai, Claesson, and Ferrari introduced the $\sigma$-machine $s \circ s_{\sigma}$ as a generalization of West's $2$-stack-sorting-map $s \circ s$. As a further generalization, in 2021, Baril, Cerbai, Khalil, and Vajnovski introduced the $(\sigma, \tau)$-machine $s \circ s_{\sigma, \tau}$ and enumerated $|\Sort_{n}(\sigma,\tau)|$---the number of permutations in $S_n$ that are mapped to the identity by the $(\sigma, \tau)$-machine---for six pairs of length $3$ permutations $(\sigma, \tau)$. In this work, we settle a conjecture by Baril, Cerbai, Khalil, and Vajnovski on the only remaining pair of length $3$ patterns $(\sigma, \tau) = (132, 321)$ for which $|\Sort_{n}(\sigma, \tau)|$ appears in the OEIS. In addition, we enumerate $|\Sort_n(123, 321)|$, which does not appear in the OEIS, but has a simple closed form. 
\end{abstract}

\maketitle

\section{Introduction}\label{intro}

In 1990, Julian West~\cite{WEST1990} introduced the stack-sorting map $s$, a deterministic variation of Knuth's~\cite{knuth1997art} stack-sorting machine. In West's stack-sorting map $s$, the input permutation is sent through a stack in a right-greedy manner such that the stack is increasing from top to bottom (see for example, \Cref{Westmap}).

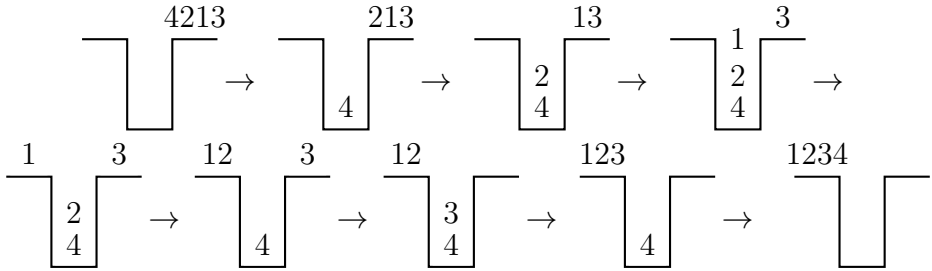
\begin{figure}
\begin{center}
\begin{tikzpicture}[scale=0.6]
				\draw[thick] (0,0) -- (1,0) -- (1,-2) -- (2,-2) -- (2,0) -- (3,0);
				\node[fill = white, draw = white] at (2.5,.5) {4213};
				\node[fill = white, draw = white] at (1.5,-1.5) {};
				\node[fill = white, draw = white] at (1.5,-.8) {};
				\node[fill = white, draw = white] at (3.5,-1) {$\rightarrow$};
              \end{tikzpicture}
              \begin{tikzpicture}[scale=0.6]
				\draw[thick] (0,0) -- (1,0) -- (1,-2) -- (2,-2) -- (2,0) -- (3,0);
				\node[fill = white, draw = white] at (2.5,.5) {213};
				\node[fill = white, draw = white] at (1.5,-1.5) {4};
                \node[fill = white, draw = white] at (1.5,-.8) {};
				\node[fill = white, draw = white] at (.5,.5) {};
             \node[fill = white, draw = white] at (3.5,-1) {$\rightarrow$}; \end{tikzpicture}
         \begin{tikzpicture}[scale=0.6]
				\draw[thick] (0,0) -- (1,0) -- (1,-2) -- (2,-2) -- (2,0) -- (3,0);
				\node[fill = white, draw = white] at (2.5,.5) {13};
				\node[fill = white, draw = white] at (1.5,-1.5) {4};
				\node[fill = white, draw = white] at (1.5,-.8) {2};
				\node[fill = white, draw = white] at (3.5,-1) {$\rightarrow$};
              \end{tikzpicture} 
              \begin{tikzpicture}[scale=0.6]
				\draw[thick] (0,0) -- (1,0) -- (1,-2) -- (2,-2) -- (2,0) -- (3,0);
				\node[fill = white, draw = white] at (2.5,.5) {3};
				\node[fill = white, draw = white] at (1.5,-1.5) {4};
				\node[fill = white, draw = white] at (1.5,-.8) {2};
                \node[fill = white, draw = white] at (1.5,0) {1};
             	\node[fill = white, draw = white] at (3.5,-1) {$\rightarrow$}; \end{tikzpicture}
              \begin{tikzpicture}[scale=0.6]
				\draw[thick] (0,0) -- (1,0) -- (1,-2) -- (2,-2) -- (2,0) -- (3,0);
				\node[fill = white, draw = white] at (2.5,.5) {3};
				\node[fill = white, draw = white] at (1.5,-1.5) {4};
                 \node[fill = white, draw = white] at (1.5,-.8) {2};
				\node[fill = white, draw = white] at (.5,.5) {1};
             	\node[fill = white, draw = white] at (3.5,-1) {$\rightarrow$}; \end{tikzpicture}\begin{tikzpicture}[scale=0.6]
				\draw[thick] (0,0) -- (1,0) -- (1,-2) -- (2,-2) -- (2,0) -- (3,0);
				\node[fill = white, draw = white] at (2.5,.5) {3};
				\node[fill = white, draw = white] at (1.5,-1.5) {4};
                 \node[fill = white, draw = white] at (1.5,-.8) {};
                 \node[fill = white, draw = white] at (1.5,-.1) {};
				\node[fill = white, draw = white] at (.5,.5) {12};
             	\node[fill = white, draw = white] at (3.5,-1) {$\rightarrow$}; \end{tikzpicture}\begin{tikzpicture}[scale=0.6]
				\draw[thick] (0,0) -- (1,0) -- (1,-2) -- (2,-2) -- (2,0) -- (3,0);
				\node[fill = white, draw = white] at (2.5,.5) {};
				\node[fill = white, draw = white] at (1.5,-1.5) {4};
                 \node[fill = white, draw = white] at (1.5,-.8) {3};
				\node[fill = white, draw = white] at (.5,.5) {12};
             	\node[fill = white, draw = white] at (3.5,-1) {$\rightarrow$}; \end{tikzpicture}\begin{tikzpicture}[scale=0.6]
				\draw[thick] (0,0) -- (1,0) -- (1,-2) -- (2,-2) -- (2,0) -- (3,0);
				\node[fill = white, draw = white] at (2.5,.5) {};
				\node[fill = white, draw = white] at (1.5,-1.5) {4};
                 \node[fill = white, draw = white] at (1.5,-.8) {};
				\node[fill = white, draw = white] at (.5,.5) {123};
             	\node[fill = white, draw = white] at (3.5,-1) {$\rightarrow$}; \end{tikzpicture}
              \begin{tikzpicture}[scale=0.6]
				\draw[thick] (0,0) -- (1,0) -- (1,-2) -- (2,-2) -- (2,0) -- (3,0);
				\node[fill = white, draw = white] at (.5,.5) {1234};\end{tikzpicture}
\end{center}
\caption{West's stack-sorting map $s$ on $x = 4213$.}
\label{Westmap}
\end{figure}

Since then, the stack-sorting map $s$ has been generalized to $s_{\sigma}$ for permutations $\sigma$~\cite{CERBAI2020105230}. Like $s$, the map $s_{\sigma}$ sends a permutation through a stack in a right-greedy manner. However, instead of insisting that the stack increases (i.e., avoids subsequences that are order-isomorphic to $21$ when as usual, is read from top to bottom), we insist that the stack avoids subsequences that are order-isomorphic to $\sigma$. Soon after, Berlow~\cite{BERLOW2021112571} extended $s_{\sigma}$ to stacks that must avoid all of the patterns in a given set. In the same year, Defant and Zheng~\cite{DEFANT2021102192} extended $s_{\sigma}$ to vincular patterns. 

West's stack-sorting map $s$ derives its name from the easily verifiable result that for any $x \in S_n$, the permutation $s^{n-1}(x)$ is \emph{sorted} (i.e., is the identity permutation). Naturally, one of the most classical questions on the stack-sorting map concerns the characterization of \emph{$1$-stack-sortable} permutations---the permutations $x \in S_n$ for which $s(x)$ is sorted. Knuth~\cite{knuth1997art} was the first to answer this question; he showed that $s(x)$ is sorted if and only if $x$ avoids subsequences that are order-isomorphic to $231$ and enumerated the number of such permutations in $S_n$ to be $n^{\text{th}}$ Catalan number $\frac{1}{n+1} \binom{2n}{n}$.

With the question on the permutations that get sorted after passing through one stack settled, it is then natural to ask about the permutations that are sorted after passing through multiple stacks in a series. A notable example is \emph{West's $2$-stack-sorting-map} $s \circ s$. In 1990, West~\cite{WEST1990} characterized the $2$-stack-sortable permutations---the permutations that are sorted after passing through West's $2$-stack-sorting-map---using barred patterns and conjectured that the number of $2$-stack-sortable permutations in $S_n$ is $\frac{2}{(n+1)(2n+1)} \binom{3n}{n}$. Two years later in 1992, Zeilberger~\cite{ZEILBERGER19928592} affirmed West's conjecture. Later, Cerbai~\cite{CERBAI2021322341} generalized some of the results on West's $2$-stack-sorting map to Cayley permutations. Only recently in 2020, Defant presented a polynomial time algorithm \cite{DEFANT2021} that counts the number of permutations that are sorted after passing through a series of three stacks.

To better understand the permutations that are sortable under stacks in series, Cerbai, Claesson, and Ferrari~\cite{CERBAI2020105230} introduced the $\sigma$-machine $s \circ s_{\sigma}$ in 2020. It is then natural to ask about the permutations that get sorted after passing through the $\sigma-$machine once. Therefore, let $\Sort_{n}(\sigma)$ denote the set of permutations in $S_n$ that are sorted after a single pass through the $\sigma$-machine. Equivalently, due to Knuth's characterization of $1$-stack sortable permutations, $\Sort_n(\sigma)$ is the set of permutations $x$ where $s_{\sigma}(x)$ avoids $231$. The most notable enumerations of the sequence $|\Sort_n(\sigma)|$ are by Cerbai, Claesson, and Ferrari~\cite{CERBAI2020105230} and by Cerbai, Claesson, Ferrari, and Steingrímsson~\cite{CERBAI2020P3.32} in a subsequent work (see \Cref{sortsigma}). 

\begin{table}
\centering
\def\arraystretch{1.25}
\begin{tabular}{p{15mm}ll}
	\toprule
	$\sigma$ & Reference & OEIS entry for~$|\Sort_{n}(\sigma)|$\\
	\midrule
    123 & Section 5 of~\cite{CERBAI2020105230} & A294790\\
    132 & Corollary 4.17 of~\cite{CERBAI2020P3.32} & A007317\\
    321 & Section 4 of~\cite{CERBAI2020105230} & A011782\\
%    3214
%    \newline 4213
%    \newline 4312
%    \newline 4321 & Section 3 of~\cite{CERBAI2020105230} & A001519\\ 
    \bottomrule
    \end{tabular}
    \caption{Enumeration of $|\Sort_{n}(\sigma)|$, for patterns $\sigma$ of length three.}
    \label{sortsigma} 
\end{table}

In 2021, Baril, Cerbai, Khalil, and Vajnovszki~\cite{BARIL2021106138} introduced the $(\sigma, \tau)$-machine $s \circ s_{\sigma, \tau}$ as a further generalization of the $\sigma$-machine. For example, \Cref{machine} illustrates that $x = 2314$ is sent to $3124$ after passing through the $(132, 321)$-machine. 

\begin{figure}
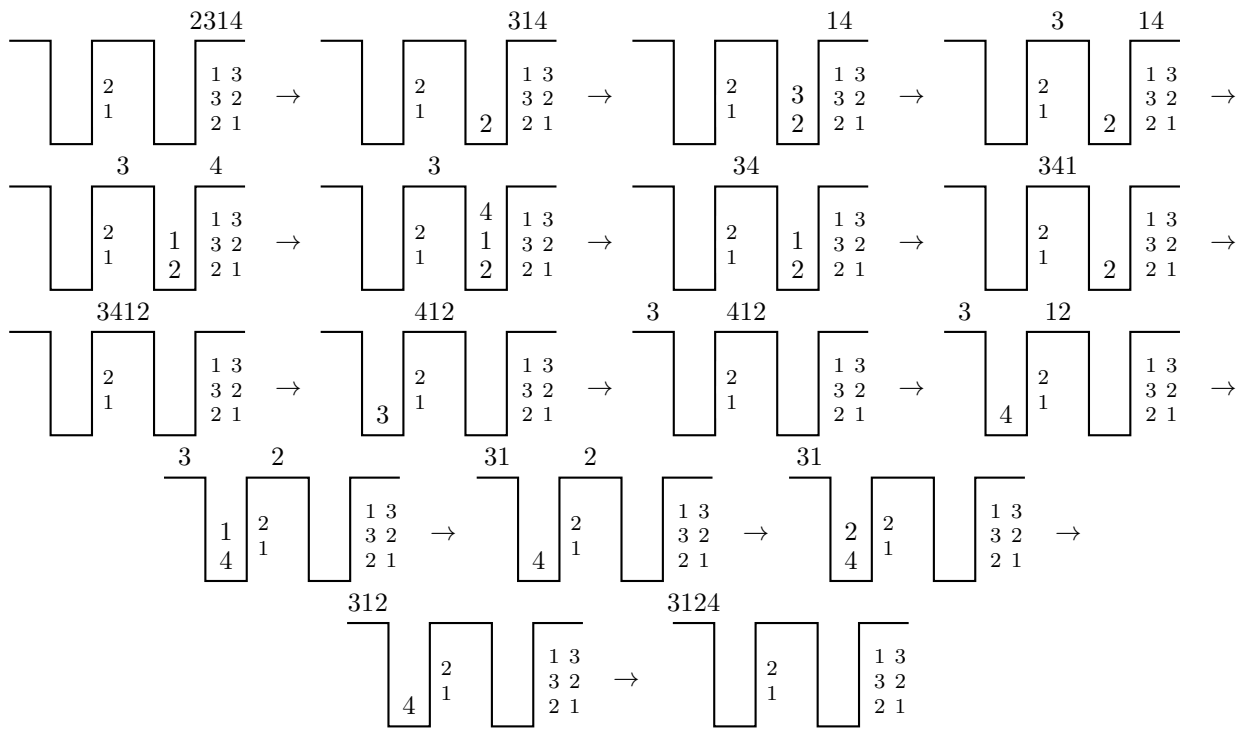

\begin{center}
\footnotesize
\machine{2314}{}{}{}{}{}{}{}{}
\machine{314}{}{}{2}{}{}{}{}{}
\machine{14}{}{3}{2}{}{}{}{}{}
\machine{14}{}{}{2}{3}{}{}{}{}
\machine{4}{}{1}{2}{3}{}{}{}{}
\machine{}{4}{1}{2}{3}{}{}{}{}
\machine{}{}{1}{2}{34}{}{}{}{}
\machine{}{}{}{2}{341}{}{}{}{}
\machine{}{}{}{}{3412}{}{}{}{}
\machine{}{}{}{}{412}{}{}{3}{}
\machine{}{}{}{}{412}{}{}{}{3}
\machine{}{}{}{}{12}{}{}{4}{3}
\machine{}{}{}{}{2}{}{1}{4}{3}
\machine{}{}{}{}{2}{}{}{4}{31}
\machine{}{}{}{}{}{}{2}{4}{31}
\machine{}{}{}{}{}{}{}{4}{312}
\machineend{}{}{}{}{}{}{}{}{3124}
\end{center}
\caption{The $(132, 321)$-machine on $x = 2314$.}
\label{machine}
\end{figure}

As an analog of $\Sort_{n}(\sigma)$, let $\Sort_{n}(\sigma, \tau)$ consist of the permutations in $S_n$ that are sorted after an iteration through the $(\sigma, \tau)$-machine. Permutations in $\Sort_n(\sigma,\tau)$ are said to be \emph{$(\sigma,\tau)$-sortable}. Baril, Cerbai, Khalil, and Vajnovszki~\cite{BARIL2021106138} enumerated $|\Sort_{n}(\sigma, \tau)|$ for several pairs of length $3$ patterns $(\sigma, \tau)$ (see \Cref{iplenum}) and observed through numerical computation that the only remaining pair of length $3$ permutations $(\sigma, \tau)$ for which $|\Sort_{n}(\sigma, \tau)|$ may belong to the OEIS~\cite{oeis} is $(\sigma, \tau) = (132, 321)$. As our first main result, in \Cref{section_132_321} we settle Baril, Cerbai, Khalil, and Vajnovszki's conjecture that $|\Sort_{n}(132, 321)|$ is enumerated by A102407 (see Conjecture 1 of~\cite{BARIL2021106138}).

\begin{table}
\centering
\def\arraystretch{1.25}
\begin{tabular}{p{15mm}llr}
	\toprule
	$(\sigma,\tau)$ &  Reference & Sequence~$|\Sort_{n}(\sigma,\tau)|$ & OEIS\\
	\midrule
	123,213\newline
	132,312\newline
	231,321 & Theorem 4.2 of~\cite{BARIL2021106138} & Catalan numbers & A000108\\
	123,132 & Corollary 5.3 of~\cite{BARIL2021106138} & Catalan numbers & A000108\\
	123,231 & Corollary 3.2 of~\cite{BARIL2021106138} & Large Schr\"oder numbers & A006318\\
	123,312 & Section 6 of~\cite{BARIL2021106138} & Binomial Transform of Catalan numbers & A007317 \\
	132,321 & \Cref{main} & 1, 2, 4, 10, 26, 72, 206, 606,$\dots$ & A102407 \\
	123,321 & \Cref{secondary} & $2^{n-1}$ for $n \leq 3$ and $7 \cdot 2^{n-4}$ for $n \geq 4$ &\\
	\bottomrule
\end{tabular}
\caption{Enumerations of $|\Sort_n(\sigma, \tau)|$.}
\label{iplenum}
\end{table}

\begin{theorem} \label{main}
	The sequence $|\Sort_{n}(132,321)|$ is the OEIS A102407~\cite{oeis}.
\end{theorem}

To prove \Cref{main}, we first characterize $\Sort_n(132,321)$ by avoidance of two (generalized) patterns of length three. We then use a bijection first discovered by West to map $\Sort_n(132,321)$ to a set that we are able to enumerate directly. A more geometrical approach that involves a bijection with Dyck paths avoiding the subword $dudu$ is also illustrated in \Cref{Dyck}.

Furthermore, in \Cref{section_123_321}, we enumerate $\Sort_n(123, 321)$. This sequence does not belong to the OEIS~\cite{oeis} but has a simple closed form for $n \geq 4$. 

\begin{theorem} \label{secondary}
   The sequence $|\Sort_{n}(123, 321)|$ is the sequence $7 \cdot 2^{n-4}$ for $n \geq 4.$
\end{theorem}

We end the paper with some suggestions for future work.

\section{Preliminaries}\label{prelim}

%\subsection{Definitions and a preliminary result}\label{definitions}

Let $[n]=\{1,2,\dots,n\}$. A \emph{permutation} of length $n$ is a bijective map $x:[n]\to[n]$. Let $S_n$ be the set of permutations of length $n$. We denote a permutation $x\in S_n$ by the string $x=x_1 x_2 \dots x_n$, where $x_i=x(i)$. Let $\indx_x(i)=x^{-1}(i)$ be the index such that $x_{\indx_{x}(i)}=i$. For example, $\indx_{12534}(5)=3$. If $i<j$ (respectively, $i>j$), we say that $x_i$ \emph{appears to the left of (resp. to the right of)} $x_j$ in $x$. An entry $x_i$ is a \emph{left-to-right minimum} (briefly, ltr minimum) of $x$ if $x_i$ is smaller than each entry that appears to its left. Furthermore, we let $x_{[i, j]}$ denote the subsequence $x_{i}x_{i+1}\dots x_{j}$, and we let $\{x_{[i, j]}\}$ be the set $\{x_{i},x_{i+1},\dots, x_{j}\}$. For $k\ge 1$, we also let $x+k=(x_1+k)(x_2+k)\dots (x_n+k)$ denote the string obtained by adding $k$ to each entry of $x$. As an example to illustrate the above definitions, let $x=524613$. The ltr-minima of $x$ are $5$, $2$, and $1$. For $i=2$ and $j=5$, we have $x_{[i,j]}=2461$ and $\{x_{[i,j]}\}=\{2,4,6,1\}$. Finally, we have $x+3=857946$.

A short introduction to permutation patterns can be found in Bevan's note~\cite{BEVAN2015}. Let us recall some notions that will be used in this paper.
%Now, suppose that $y$ and $z$ are distinct subsequences of $x$ of the same length. Then we say that $y$ is \emph{lexicographically smaller} than $z$ if $y_i < z_i$ for the smallest $i$ such that $y_i \ne z_i$. Otherwise, $y$ is \emph{lexicographically larger}. For example, if $x = 12543$, then $y = 1543$ is lexicographically larger than $z = 1253$.
Given two permutations $x\in S_n$ and $y\in S_m$, with $m\le n$, we say that $x$ \emph{contains} the pattern $y$ if there exists some $c(1) < c(2) < \dots < c(m)$ such that the subsequence $z = x_{c(1)}x_{c(2)} \dots x_{c(m)}$ is order-isomorphic to $y$. If so, we say that $z$ is an \emph{occurrence} of $y$ in $x$. If $x$ does not contain the pattern $y$, then $x$ \emph{avoids} $y$. Let $\Av_n(y)$ denote the set of permutations of length $n$ that avoid $y$. Similarly, for a set of patterns $T$ we let $\Av_{n}(T)$ be the set of permutations in $S_n$ avoiding all the patterns in $T$. Bivincular patterns~\cite{BMCDK2010} are obtained by allowing constraints of adjacency on positions and values. Mesh patterns generalize both classical and bivincular patterns in the following sense. A \emph{mesh pattern}~\cite{BC2011} is a pair $(y,R)$, where $y\in S_m$ is a permutation and $R \subseteq \left[ 0,m \right] \times \left[ 0,m \right]$ is a set of pairs of integers. Each pair in $R$ determines the lower left corners of unit squares which specify forbidden regions in the plot of permutations. An occurrence of the mesh pattern $(y,R)$ in the permutation $x$ is an occurrence of the classical pattern $y$ such that no other points of $x$ occur in the forbidden regions specified by $R$. For the rest of this paper, we shall denote by $123^{\star}$ and $132^{\star}$ the bivincular patterns depicted--as mesh patterns--in \Cref{mesh_patterns}. For example, an occurrence of $123^{\star}$ in a permutation $x$ is a classical occurrence $x_{c(1)}x_{c(2)}x_{c(3)}$ of $123$ such that $c(3)=c(2)+1$ and $x_{c(3)}=x_{c(2)}+1$.
%For example, if $x = 12543$, then $z = 154$ is a $y = 132$ pattern of $x$.   Furthermore, let $z = x_{c(1)} x_{c(2)} x_{c(3)}$ be a $132^{\star}$ pattern of $x$ if $c(3) = c(2) + 1$ and $x_{c(3)} = x_{c(2)}-1$. For example, if $x = 12543$, then $z =154$ is a $132^{\star}$ pattern of $x$, but $z=153$ is not. Similarly, let $z = x_{c(1)} x_{c(2)} x_{c(3)}$ be a $123^{\star}$ pattern of $x$ if $c(3) = c(2) + 1$ and $x_{c(3)} = x_{c(2)}+1$. For example, if $x = 12453$, then $z=145$ is a $123^{\star}$ pattern of $x$, but $z= 123$ is not.

\begin{figure}
$$
123^{\star} \;=\;
\begin{tikzpicture}[scale=0.40, baseline=20.5pt]
\fill[NE-lines] (2,0) rectangle (3,4);
\fill[NE-lines] (0,2) rectangle (4,3);
\draw [semithick] (0.001,0.001) grid (3.999,3.999);
\filldraw (1,1) circle (6pt);
\filldraw (2,2) circle (6pt);
\filldraw (3,3) circle (6pt);
\end{tikzpicture}
\qquad
132^{\star} \;=\;
\begin{tikzpicture}[scale=0.40, baseline=20.5pt]
\fill[NE-lines] (2,0) rectangle (3,4);
\fill[NE-lines] (0,2) rectangle (4,3);
\draw [semithick] (0.001,0.001) grid (3.999,3.999);
\filldraw (1,1) circle (6pt);
\filldraw (2,3) circle (6pt);
\filldraw (3,2) circle (6pt);
\end{tikzpicture}
$$
\caption{Bivincular patterns $123^{\star}$ and $132^{\star}$.}\label{mesh_patterns}
\end{figure}
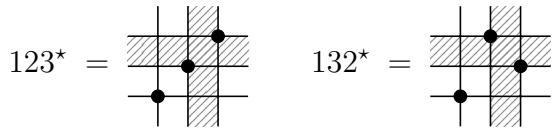

Lastly, we cite Knuth's~\cite{knuth1997art} lemma on the characterization of 1-stack-sortable permutations. 

\begin{lemma}[Knuth~\cite{knuth1997art}]\label{knuth}
 A permutation $x$ is 1-stack-sortable if and only if $x$ avoids $231$. 
 %It holds that $x \in \Av_n(231)$ if and only if $x \in S_n$ and $s(x)$ is sorted.
\end{lemma}

\subsection{West's Bijection}\label{west_bij}

In 1995, West~\cite{WEST1995247} presented a new bijection between $\Av_n(132)$ and $\Av_n(123)$. Let $x = x_1 x_2 \dots x_{n-1} \in S_{n-1}$. For $1 \leq i \leq n$, denote by $x^{i}$ the permutation $x_{1} x_{2} \dots x_{i-1} n x_{i} \dots x_{n-1}$ obtained by inserting $n$ in the $i$th \emph{site} of $x$ (i.e. the position immediately before $x_i$, for $i=1,\dots,n-1$, or at the end, if $i=n$). Given a pattern $y$, define the $i$th site of $x$ to be \emph{active} with respect to $y$ if $x^i$ avoids $y$. For example, if $x = 45231$, then $i \in \{1,3,5,6\}$ are active with respect to $y =132$, but $i \in \{2,4\}$ are not since both $x^2=465231$ and $x^4=452631$ contain $132$. Now, let $\smallx_k(x)$ denote the subsequence of the smallest $k$ numbers in $x$. Given a pattern $y$, let $\Act_{j}(x; y)$ denote the set of active sites in $\smallx_{|x|+1-j}(x)$ with respect to the pattern $y$. Note that $\Act_1(x; y)$ is equal to the number of active sites of $x$ with respect to $y$. For example, $|\Act_{1}(45231;132)|  = |\Act_{2}(45231; 132)|= 4$, because the set of active sites with respect to $132$ pattern for $x=45231$ is $\{1,3,5,6\}$ and for $\smallx_{4}(x) = 4231$ is $\{1,2,4,5\}$. The \emph{signature} of $x$ with respect to $y$ is the word
$$
\mathfrak{S}(x; y)=|\Act_{1}(x; y)| \cdot |\Act_{2}(x; y)|  \dots |\Act_{n-1}(x; y)|,
$$
where $\cdot$ denotes concatenation. For example, $\mathfrak{S}(45231; 132) = 44332$. In addition, for a set of permutations $T$, let $\mathfrak{S}(T; y) := \{\mathfrak{S}(x; y) \mid x \in T\}$. 
West\cite{WEST1995247} showed that permutations in $\Av(123)$, as well as permutations in $\Av(132)$, are uniquely determined by their signature. Furthermore, he showed that $\mathfrak{S}(\Av_{n}(132); 132) = \mathfrak{S}(\Av_{n}(123); 123)$. This naturally induces a bijection between $\Av_{n}(132)$ and $\Av_{n}(123)$. For example, $x = 45231$ is the only $x \in \Av_5(132)$ such that $\mathfrak{S}(x; 132) = 44332,$ and $y = 42153$ is the only $y \in \Av_5(123)$ such that $\mathfrak{S}(y; 123) = 44332$.

%\section{Proofs of the Main Results}\label{proofs}
%\section{Proof of \Cref{main}}
\section{The $(132,321)$-machine}\label{section_132_321}

This section is devoted to the $(132,321)$-machine. First we prove that $\Sort_n(132, 321)=\Av_n(123,132^{\star})$. Then we show that West's bijection maps $\Av_n(123,132^{\star})$ to $\Av_n(132,123^{\star})$. Finally, we enumerate $\Sort_n(132,321)$ by providing a generating function for $\Av_n(132,123^{\star})$. We end the section by sketching a bijection between $\Sort_n(132,321)$ and the set of Dyck paths of semilength $n$ that avoid $dudu$ as a factor.

Let us begin by establishing a property that all ltr minima of the input permutation must satisfy.

\begin{lemma} \label{ltr}
   If $x_i$ is a ltr minimum of $x$ and $\indx_{s_{132,321}(x)}(x_i)<\indx_{s_{132,321}(x)}(x_j)$, then $j < i$.
\end{lemma}
\begin{proof}
Because $x_i$ is a ltr minimum and the stack must avoid the $132$ pattern, the stack must be increasing from top to bottom right after $x_i$ enters the stack. Now, because $z_{2} > z_{3}$ in any $132$ or $321$ pattern $z$ and the numbers below $x_i$ in the stack are increasing, $x_i$ can only leave the stack after every number has entered or left the stack. Therefore, if $x_j$ appears to the right of $x_i$ in $s_{132,321}(x)$, then $x_j$ must have been in the stack when $x_i$ entered the stack. Thus, $j < i$.
\end{proof}

Next, we show that any permutation in $\Sort_{n}(132, 321)$ must avoid the $123$ pattern.

\begin{proposition} \label{avoid_123}
    If $x \in \Sort_{n}(132, 321)$, then $x \in \Av_{n}(123)$. 
\end{proposition}
\begin{proof}
    Suppose otherwise. Let $x_{i}x_{j}x_{k}$ be the lexicographically smallest $123$ pattern of $x$ so that $x_i$ is a ltr minimum. By \Cref{ltr}, $x_i$ must be in the stack when $x_{k}$ enters the stack. Now, $x_{j}$ must leave the stack before $x_{k}$ enters the stack, because $x_{k}x_{j} x_{i}$ is a $321$ pattern. As a result, $x_{j}$ must appear to left of $x_{k}$ in $s_{132,321}(x)$. Furthermore, by \Cref{ltr}, $x_{i}$ must appear to the right of $x_{j}$ and $x_{k}$ in $s_{132,321}(x)$. Thus, $x_{j} x_{k} x_{i}$ is a $231$ pattern in $s_{132,321}(x)$---a contradiction to $x \in \Sort_{n}(132, 321)$ by \Cref{knuth}. 
\end{proof}

Due to \Cref{avoid_123}, the behavior of the $(132,321)$-stack on any $(132,321)$-sortable permutation is equivalent to the behavior of a $132$-stack. Therefore,
$$
\Sort_n(132,321)=\Sort_n(132)\cap\Av(123).
$$
The next proposition shows that any permutation in $\Sort_{n}(132, 321)$ must avoid the $132^{\star}$ pattern. 

\begin{proposition} \label{avoid_132}
    If $x \in \Sort_n(132, 321)$, then $x \in \Av_{n}(132^{\star})$. 
\end{proposition}
\begin{proof}
    Suppose otherwise. Let $x_{i}x_{j}x_{j+1}$ be the lexicographically smallest $132^{\star}$ pattern of $x$ so that $x_i$ is a ltr minimum. Now, suppose that $x_j$ has just entered the stack. Because $x_{j} = x_{j+1} +1$, if $x_{j+1}$ entering the stack introduces a $132$ or $321$ pattern to the stack, then $x_{j+1}$ and $x_{j}$ must both be a part of the pattern. However, in a $132$ pattern $z$, it must be that $z_{2} - z_{1} \geq 2$ and in a $321$ pattern $z$, it must be that $z_{2} < z_{1}$. Yet, $x_{j} - x_{j+1} = 1$, and so the entrance of $x_{j+1}$ into the stack cannot introduce a $132$ or $321$ pattern. Therefore, $x_{j+1}$ must be to the left of $x_{j}$ in $s_{132,321}(x)$. In addition, by \Cref{ltr}, $x_{i}$ must be to the right of $x_{j}$ and $x_{j+1}$ in $s_{132,321}(x)$. Thus, $x_{j+1} x_{j} x_{i}$ is a $231$ pattern in $s_{132,321}(x)$---a contradiction to $x \in \Sort_{n}(132, 321)$ by \Cref{knuth}. 
\end{proof}

We shall prove that the conditions of Lemmas \ref{avoid_123} and \ref{avoid_132} are sufficient for a permutation to be in $\Sort_{n}(132, 321)$. First, a simple lemma.

\begin{lemma} \label{ltr2}
    If $x_i$ is not a ltr minimum of $x \in \Av_{n}(123)$ and $i < j$, then $x_i > x_j$. 
\end{lemma}
\begin{proof}
    Suppose otherwise. Because $x_i$ is not a ltr minimum, there exists $x_k$ such that $x_k < x_i$ and $k < i$. Thus, $x_k x_{i} x_{j}$ is a $123$ pattern in $x$---a contradiction. 
\end{proof}

Next, we prove a property that all $x \in \Av_{n}(123, 132^{\star})$ must satisfy. 

\begin{lemma} \label{not-ltr}
    If $x_i$ is not a ltr minimum of $x \in \Av_{n}(123, 132^{\star})$, then 
    $$\indx_{s_{132,321}(x)}(x_i) < \indx_{s_{132,321}(x)}(x_{i+1}).$$  
\end{lemma}
\begin{proof}
    First, suppose that $x_{i+1}$ is a ltr minimum. Then because $x_i$ is not a ltr minimum, there must be some ltr minimum $x_j$ such that $x_j < x_i$ and $j < i$. Now, by \Cref{ltr}, $x_j$ must be in the stack when $x_{i+1}$ enters the stack. Then because $x_{i+1}$ is a ltr minimum, the subsequence $x_{i+1} x_{i} x_{j}$ must be a $132$ pattern. Therefore, $x_{i}$ must leave the stack before $x_{i+1}$ enters the stack. 

    Next, suppose that $x_{i+1}$ is not a ltr minimum. First, because $x_i$ is not a ltr minimum, $x_i > x_{i+1}$ by \Cref{ltr2}. In addition, there is some ltr minimum $x_k$ such that $x_k < x_{i+1}$ and $k < i$. 
Furthermore, $x_{i} > x_{i+1} +1$, because otherwise, $x_{k}x_{i}x_{i+1}$ is a $132^{\star}$ pattern. Now, if $x_{i+1}+1$ appears to the right of $x_{i+1}$, then $x_k x_{i+1}\cdot(x_{i+1}+1)$ is a $123$ pattern. Thus, by \Cref{ltr2}, $x_{i+1}+1$ must be a ltr minimum, and by \Cref{ltr}, $x_{i+1}+1$ must be in the stack when $x_{i+1}$ enters the stack. Now, $x_{i}$ and $x_{i+1}+1$ cannot both be in the stack when $x_{i+1}$ enters the stack, because $x_{i+1} x_{i}\cdot(x_{i+1}+1)$ is a $132$ pattern. Thus, $x_{i}$ must leave the stack before $x_{i+1}$ enters the stack.
\end{proof}

%Next, we use Lemmas \ref{ltr} and \ref{not-ltr} to show that the conditions in Lemmas \ref{avoid_123} and \ref{avoid_132} are sufficient for a permutation to be in $\Sort_{n}(132, 321)$.

\begin{proposition}
 \label{sufficent}
    If $x \in \Av_n(123,132^{\star})$, then $x \in \Sort_n(132, 321).$
\end{proposition}
\begin{proof} 
    By \Cref{ltr}, any ltr minimum $x_i$ must remain in the stack until every other number has entered the stack. By \Cref{not-ltr}, any $x_i$ that is not a ltr minimum of $x$ must leave the stack before $x_{i+1}$ enters the stack. Therefore, any ltr minimum of $x$ must appear to the right of any number that is not a ltr minimum of $x$ in $s_{132,321}(x)$. Now, because the ltr minima of $x$ must decreasing from the left to right, by \Cref{ltr}, the ltr minima of $x$ must appear in $s_{132,321}(x)$ in increasing order. Furthermore, by Lemmas \ref{ltr2} and \ref{not-ltr}, the numbers that are not ltr minima of $x$ must appear in $s_{132,321}(x)$ in decreasing order. Thus, $s_{132,321}(x)$ must avoid the $231$ pattern and by \Cref{knuth}, must be in $\Sort_{n}(132,321)$. 
\end{proof}

Now, we collate propositions \ref{avoid_123}, \ref{avoid_132}, and \ref{sufficent}.

\begin{theorem} \label{char}
    For all $n$, it holds that $\Sort_{n}(132, 321) = \Av_{n}(123, 132^{\star})$. 
\end{theorem}

Our next goal is to show that West's bijection maps $\Av_n(123,132^{\star})$ to $\Av_{n}(132, 123^{\star})$. Recall from \Cref{west_bij} that permutations in $\Av_n(123)$, as well as in $\Av(132)$, are uniquely determined by their signature. West's bijection then follows from the equality between the corresponding sets of signatures
$$
\mathfrak{S}(\Av_n(123),123)=\mathfrak{S}(\Av_n(132),132).
$$ 
We wish to prove that any permutation $x\in\Av_n(123)$ contains $132^{\star}$ if and only if there exists some $1\le i\le n-2$ such that $$
\mathfrak{S}(x; 123)_{i} = \mathfrak{S}(x; 123)_{i+1} \leq \mathfrak{S}(x; 123)_{i+2}.
$$
The desired result follows from the fact that the same property holds when we replace the pattern $123$ with $132$ and the pattern $132^{\star}$ with $123^{\star}$.

Let us take care of $\Av_n(123,132^{\star})$ first. To start, we show that removing the maximum of a permutation preserves the property of avoiding both patterns.

\begin{lemma} \label{induct}
    If $x \in \Av_{n}(123, 132^{\star})$, then $\smallx_{n-1}(x) \in \Av_{n-1}(123, 132^{\star})$. 
\end{lemma}
\begin{proof}
    It is clear that $\smallx_{n-1}(x) \in \Av_{n-1}(123)$. Suppose that $\smallx_{n-1}(x) \not \in \Av_{n-1}(132^{\star})$. Then $x_j x_{\mathrm{ind_{x}(n)}-1} x_{\mathrm{ind_{x}(n)}+1}$ must be a $132^{\star}$ pattern in $\smallx_{n-1}(x)$, because $x \in \Av_{n}(132^{\star})$. However, then $x_{j} x_{\indx_{x}(n)-1} \cdot n$ is a $123$ pattern in $x$---a contradiction to $x \in \Av_{n}(123, 132^{\star})$. 
\end{proof}

Now, let $x\in\Av_n(123)$. It is easy to see that the set of active sites $\Act_1(x;123)$ of $x$ with respect to $123$ forms an interval. The entry $x_{|\Act_{1}(x; 123)|}$ preceding the leftmost site that is not active can be characterized as follows.
%Let us now give a formula for the number $|\Act_{n}(x;123)|$ of active sites in $\smallx_{n+1}$ when $x$ avoids $123$.

\begin{lemma} \label{signature_char}
    Let $x \in \Av_{n}(123)$ and suppose that $x$ is not the decreasing permutation $n(n-1)\dots 1$. Then $x_{|\Act_{1}(x; 123)|} = \max \{ x_i \mid x_{i} + i \ne n+1 \}.$
\end{lemma}
\begin{proof}
    We induct on $n.$ The statement is vacuously true for $n = 1.$ First, suppose that $x_1 \ne n$. Clearly, $x^j \in \Av_n(123)$ for $1 \leq j \leq \indx_x(n),$ since if $(x^j)_a \cdot (x^j)_b \cdot (n+1)$ is a $123$ pattern in $x^j,$ then $x_a x_b n$ is a $123$ pattern in $x$. Conversely, if $j > \indx_x(n)$, then $x^j_1 \cdot n \cdot (n+1)$ is a $123$ pattern in $x^j,$ and so $j \not \in \Act_{1}(x; 123)$. Hence, $\Act_{1}(x; 123) = \{1, 2, \cdots, \indx_x(n)\}$ and $x_{|\Act_1(x; 123)|} = \max \{ x_i \mid x_{i} + i \ne n+1 \} = n$.

    Now, suppose that $x_{1} = n$. Then $\Act_{1}(x;123) = \{j\mid j-1 \in \Act_2(x;123) \} \cup \{1\}$, because no $123$ pattern in $x^{j}$ for $1 \leq j \leq n+1$ begins with $n$ or $n+1$. Thus, $|\Act_{1}(x;123)| = |\Act_2(x;123)| + 1$, and the statement of the lemma follows by the inductive hypothesis.
\end{proof}

%Next, we prove a property of $\mathfrak{S}_n(x; 123)$ for $x \in \Av_n(123) \setminus \Av_n(132^{\star})$. 

\begin{proposition} \label{necessary}
    If $x \in \Av_n(123)$ and $|\Act_{i}(x; 123)| = \mathfrak
|\Act_{i+1}(x; 123)| \leq |\Act_{i+2}(x; 123)|$ for some $1 \leq i \leq n-2$, then $x \not\in \Av_n(132^{\star}).$
\end{proposition}
\begin{proof}
We induct on $n$; it is easy to see that the statement holds for $n \leq 3$. Now, because $\mathfrak{S}_{n}(x; 123) = |\Act_{1}(x; 123)| \cdot \mathfrak{S}_{n}(\smallx_{n-1}(x; 123))$, if $n \geq 4$ and $2 \leq i \leq n-2$, then by \Cref{induct} and the inductive hypothesis, $x \not \in \Av_{n}(132^{\star})$. Thus, assume $i = 1$.

 Next, if $x_1 =n$, then $|\Act_1(x; 123)| = |\Act_2(x;123)| + 1$ by \Cref{signature_char}. Similarly, if $x_1 = n-1$, then $|\Act_2(x; 123)| = |\Act_3(x;123)| + 1$ by \Cref{signature_char} on $\smallx_{n-1}(x)$. Thus, $x_1 \leq n-2$. 

 Now, $n$ appears to the left of $n-1$ in $x$, because $x \in \Av_{n}(123)$ and $x_1 \leq n-2$. But if $n \cdot (n-1)$ is not a substring of $x$, then by \Cref{signature_char} and $x_1 \leq n-2$,  
 \begin{align*}
     |\Act_1(x; 123)| = \indx_x(n) < \indx_{\smallx_{n-1}(x)} (n-1) = |\Act_2(x; 123)|.
 \end{align*}
Thus, $x \not\in \Av_{n}(132^{\star})$, because $x_1 \cdot n \cdot (n-1)$ form a $132^{\star}$ pattern in $x$.
\end{proof}

The converse of \Cref{necessary} is also true, as we show below.
 
\begin{proposition} \label{sufficient2}
    If $x \in \Av_n(123) \setminus \Av_{n}(132^{\star})$, then for some $1 \leq i \leq n-2$, it holds that $|\Act_{i}(x; 123)| = 
|\Act_{i+1}(x; 123)| \leq |\Act_{i+2}(x; 123)|$. 
\end{proposition}
\begin{proof}
We induct on $n$; the statement holds for $n \leq 3,$ because $\mathfrak{S}_n(132; 123) = 222.$ Now, suppose for the sake of contradiction that $n \geq 4$ and that there is no $1\leq i \leq n-3$ such that $|\Act_{i}(x; 123)| = \mathfrak |\Act_{i+1}(x; 123)| \leq |\Act_{i+2}(x; 123)|.$ Then, since $\mathfrak{S}_{n}(x; 123) = |\Act_{1}(x; 123)| \cdot \mathfrak{S}_{n}(\smallx_{n-1}(x); 123),$ we have that $\smallx_{n-1}(x) \in \Av_{n-1}(132^{\star})$ from the inductive hypothesis. Now, $x \in \Av_{n}(123) \setminus \Av_{n}(132^{\star})$, but $\smallx_{n-1}(x) \in \Av_{n}(132^{\star})$, and so $\indx_{x}(n) = \indx_{x}(n-1)-1 > 1$.

Now, $|\Act_{1}(x; 123)| = \indx_{x}(n)$ from \Cref{signature_char}. Similarly, $|\Act_{2}(x; 123)| = \indx_{\smallx_{n-1}(x)}(n-1) = \indx_{x}(n)$ from \Cref{signature_char} on $\smallx_{n-1}(x)$. Therefore, $|\Act_{1}(x; 123)| = |\Act_{2}(x; 123)|$. Now, because $x \in \Av_{n}(123)$, the substring $x_{[1,\indx_{x}(n)-1]}$ must be decreasing. Thus, for all $1 \leq j \leq \indx_{x}(n)-1$, it must be that $x_j+j \leq n-1$. Therefore, for each $1 \leq j \leq  \indx_{x}(n)-1$, if $x_j + j \ne n-1$, then $x_j \leq n-2-j$. Thus, $n-2-j + \indx_{x}(n-2-j) \ne n-1$ and by \Cref{signature_char}, $|\Act_{3}(x; 123)| \ne j$. Thus, $|\Act_{3}(x; 123)| \geq |\Act_{1}(x; 123)| = |\Act_{2}(x; 123)|$. 
\end{proof}

Now, propositions \ref{necessary} and \ref{sufficient2} characterize $\Av_{n} (123, 132^{\star})$.

\begin{theorem} \label{thm1}
     For $x \in \Av_{n}(123)$, there exists some $1 \leq i \leq n-2$ such that $\mathfrak{S}_{n}(x; 123)_{i} = \mathfrak{S}_{n}(x; 123)_{i+1} \leq \mathfrak{S}_{n}(x; 123)_{i+2}$ if and only if $x \not\in \Av_{n}(132^{\star})$.
\end{theorem}

Let us now turn our attention to $\Av_{n}(132, 123^{\star})$. We shall replicate the approach used for $\Av_n(123,132^{\star})$ to obtain an analogous characterization of permutations in $\Av_n(132)$ that contain $123^{\star}$.

\begin{lemma} \label{induct2}
    If $x \in \Av_{n}(132, 123^{\star})$, then $\smallx_{n-1}(x) \in \Av_{n-1}(132, 123^{\star})$. 
\end{lemma}
\begin{proof}
    It is clear that $\smallx_{n-1}(x) \in \Av_{n-1}(132)$. It thus suffices to show that $\smallx_{n-1}(x) \in \Av_{n-1}(123^{\star})$. Suppose otherwise. Because $x \in \Av_{n}(123^{\star})$, the $123^{\star}$ pattern in $\smallx_{n-1}(x)$ must be $\smallx_{n-1}(x)_{j} \smallx_{n-1}(x)_{\indx_{x}(n)-1} \smallx_{n-1}(x)_{\indx_{x}(n)}$ for some $j < \indx_{x}(n)-1$. However, then $x_{j} \cdot n \cdot \smallx_{n-1}(x)_{\indx_{x}(n)}$ is a $132$ pattern in $x$---a contradiction to $x \in \Av_{n}(123, 132^{\star})$. 
\end{proof}

Given $x \in \Av_{n}(132)$, we express $\Act_{1}(x; 132)$ in terms of $\Act_{2}(x; 132)$. 

\begin{lemma} \label{formula}
    If $x \in \Av_{n}(132)$, then $$\Act_{1}(x; 132) =\{i \mid i-1 \in \Act_{2}(x; 132) \cap [\indx_{x}(n), n]\} \cup \{ 1\}.$$
\end{lemma}
\begin{proof}
    It is clear that $i \in \Act_{1}(x; 132)$ if and only if $\{x_{[1, i-1]}\} = [n-i+2,n]$. Thus, if $i \leq \indx_{x}(n)$, then $i \in \Act_{1}(x; 132)$ if and only if $i = 1$. But, if $\indx_{x}(n) + 1 \leq i$, then $\{x_{[1,i-1]}\} = [n-i+2, n]$ if and only if $\{\smallx(x)_{[1, i-2]}\} = [n-i+2, n-1]$ if and only if $ i-1 \in \Act_{2}(x; 132)$. 
\end{proof}

The following results are the analogs of \Cref{necessary} and \Cref{necessary2} for the $132$ pattern. 

\begin{proposition} \label{necessary2}
 If $x \in \Av_n(132)$ and $|\Act_{i}(x; 132)| = 
|\Act_{i+1}(x; 132)| \leq |\Act_{i+2}(x; 132)|$ for some $1 \leq i \leq n-2$, then $x \not\in \Av_n(123^{\star}).$
\end{proposition}
\begin{proof}
We induct on $n$; it is easy to see that the statement holds for $n \leq 3$. Now, because $\mathfrak{S}_{n}(x; 132) = |\Act_{1}(x; 132)| \cdot \mathfrak{S}_{n}(\smallx_{n-1}(x; 132))$, if $n \geq 4$ and $2 \leq i \leq n-2$, then by \Cref{induct2} and the inductive hypothesis, $x \not \in \Av_{n}(123^{\star})$. Thus, assume $i = 1$.

Now, if $x_1 = n,$ then by \Cref{formula}, $|\Act_{1}(x; 132)| = 
|\Act_{2}(x; 132)| +1$. Similarly, if $x_1 = n-1,$ then $|\Act_2(x; 132)| = |\Act_3(x; 132)| +1$ by \Cref{formula} on $\smallx_{n-1}(x)$. Thus, $x_1 \leq n-2.$ Now, $n$ appears to the right of $n-1$ in $x$, because $x \in \Av_n(132)$ and $x_1 \leq n-2$. Furthermore, $x \in \Av_{n}(132)$ implies that $\{x_{[1, \indx_{x}(n-1)]}\} = [n-\indx_{x}(n-1),n-1]$ and $\{x_{[1, \indx_{x}(n)]}\} = [n-\indx_{x}(n),n]$. Thus, if $(n-1) \cdot n$ is not a substring of $x,$ then $\{\indx_{x}(n-1)+1, \indx_{x}(n)\} \in \Act_2(x; 132)$, from which it follows that $\Act_1(x; 132) \leq \Act_2(x; 132) - 1$ by \Cref{formula}. Thus, $x \notin \Av_n(123^{\star}),$ because $x_1 \cdot (n-1) \cdot n$ is a $123^{\star}$ pattern of $x.$
\end{proof}

\begin{proposition} \label{sufficient3}
    If $x \in \Av_n(132) \setminus \Av_{n}(123^{\star})$, then for some $1 \leq i \leq n-2$, it holds that $|\Act_{i}(x; 132)| = \mathfrak
|\Act_{i+1}(x; 132)| \leq |\Act_{i+2}(x; 132)|$. 
\end{proposition}
\begin{proof}
    We induct on $n$; the statement holds for $n \leq 3,$ because $\mathfrak{S}_n(123; 132) = 222.$ Now, suppose that $n \geq 4$ and suppose for the sake of contradiction that there is no $i$ such that $|\Act_{i}(x; 132)| = \mathfrak |\Act_{i+1}(x; 132)| \leq |\Act_{i+2}(x; 132)|.$ Then, since $\mathfrak{S}_{n}(x; 132) = |\Act_{1}(x; 132)| \cdot \mathfrak{S}_{n}(\smallx_{n-1}(x); 132),$ we have that $\smallx_{n-1}(x) \in \Av_{n-1}(123^{\star})$ from the inductive hypothesis. Now, $x \in \Av_{n}(132) \setminus \Av_{n}(123^{\star})$ but $\smallx_{n-1}(x) \in \Av_{n}(123^{\star})$, and so $\indx_{x}(n-1) = \indx_{x}(n)-1 > 1$.

    Now, because $\indx_{x}(n-1) = \indx_{x}(n)-1$, it holds that $\Act_{2}(x; 132) \setminus [\indx_{x}(n), n] = \{1\}$. Therefore, $|\Act_{1}(x; 132)| =|\{i \mid i-1 \in \Act_{2}(x; 132) \cap [\indx_{x}(n), n]\}| +1 = |\Act_{2}(x; 132)|$ by \Cref{formula}. Furthermore, because $\indx_{x}(n-1) > 1 $, it follows from \Cref{formula} that $|\Act_{2}(x; 132)|=|\{i \mid i-1 \in \Act_{3}(x; 132) \cap [\indx_{x}(n-1), n-1]\} \cup \{1\}| \leq |\Act_{3}(x; 132)|$. Now, $|\Act_{1}(x; 123)| = 
|\Act_{2}(x; 123)| \leq |\Act_{3}(x; 123)|$---a contradiction as sought. 
\end{proof}

Once again, we obtain a characterization of $\Av_n(132, 123^{\star})$.

\begin{theorem} \label{thm2}
For $x \in \Av_{n}(132)$, there exists some $1 \leq i \leq n-2$ such that $\mathfrak{S}_{n}(x; 132)_{i} = \mathfrak{S}_{n}(x; 132)_{i+1} \leq \mathfrak{S}_{n}(x; 132)_{i+2}$ if and only if $x \not\in \Av_{n}(123^{\star})$.
\end{theorem}

As an immediate corollary of \Cref{thm1} and \Cref{thm2}, West's bijection maps $\Av_{n}(123, 132^{\star})$ to $\Av_{n}(132, 123^{\star})$. We state this in the following theorem. 

\begin{theorem} \label{bijection}
    West's bijection~\cite{WEST1995247} bijects $\Av_{n}(123, 132^{\star})$ and $\Av_{n}(132, 123^{\star})$. Therefore,
    $$
    |\Av_{n}(123, 132^{\star})|=|\Av_{n}(132, 123^{\star})|.
    $$
\end{theorem}

We are now able to give a generating function for the sequence $|\Av_{n}(132, 123^{\star})|$. Let $$f_{n} = |\{x \in \Av_{n}(132) \mid \indx_{x}(n)-1 = \indx_{x}(n-1) > 1, \smallx_{n-1}(x) \in \Av_{n-1}(123^{\star})\}|,$$ and let $g_{n}  = |\Av_{n}(132, 123^{\star})|$.
%We first show that $f_n$ is the difference between $g_{n-1}$ and $g_{n-2}$.

\begin{lemma} \label{recursion1}
    For all $n\geq 2$, it holds that $f_n = g_{n-1}-g_{n-2}$. 
\end{lemma}
\begin{proof}
It is straightforward from the definitions of $f_n$ and $g_n$ that 
\begin{align*}
f_n &=  |\{x \in \Av_{n}(132)  \mid \indx_{x}(n)-1 = \indx_{x}(n-1), \smallx_{n-1} (x) \in \Av_{n-1}(123^{\star})\}| \\ & \hspace{20mm} - |\{x \in \Av_{n}(132)  \mid x_1 = n-1, x_2 = n, \smallx_{n-2}(x) \in \Av_{n-2}(123^{\star})\}|  \\ &= |\{x\in \Av_{n-1}(132)  \mid x \in \Av_{n-1}(123^{\star})\}| - |\{x\in \Av_{n-2}(132)  \mid  x \in \Av_{n-2}(123^{\star})\}| \\ &= |\Av_{n-1}(132, 123^{\star})| - |\Av_{n-2}(132, 123^{\star})| =g_{n-1}-g_{n-2}. \qedhere
\end{align*}
\end{proof}
Next, we give another recursive relation between $f_n$ and $g_n$.
\begin{lemma} \label{recursion2}
    For $n \geq 2$, it holds that $g_{n-1} g_{0} + g_{n-2} g_{1} + \dots g_{0}g_{n-1} = g_{n} + f_{n}$. 
\end{lemma}
\begin{proof} For each $x \in \Av_{n}(132)$ such that $\smallx_{n-1}(x) \in \Av_{n-1}(123^{\star})$, either $x \in \Av_{n}(123^{\star})$ or $\indx_{x}(n)-1 = \indx_{x}(n-1) > 1$. Thus, 
     \[f_n + g_n = |\{x \in \Av_n(132) \mid \smallx_{n-1} (x) \in \Av_{n-1}(123^{\star})\}|.\]
    At the same time, for each $0 \leq i \leq n-1$,
    \begin{align*}
        g_{i} g_{n-1-i} &= |\{x_{[1, i]} + (i+1-n) \in \Av_{i}(132) \mid x_{[1, i]} + (i+1-n)  \in \Av_{i}(123^{\star})\}| \\ & \hspace{35mm} \cdot |\{x_{[i+2, n]} \in \Av_{n-1-i}(132) \mid x_{[i+2, n]} \in \Av_{n-1-i}(123^{\star})\}| \\ & =  |\{x \in \Av_n(132) \mid \smallx_{n-1} (x) \in \Av_{n-1}(123^{\star}), x_{i+1} = n\}|. 
    \end{align*}
Therefore, $\sum_{i=0}^{n-1} g_{i} g_{n-1-i} = |\{x \in \Av_n(132) \mid \smallx_{n-1} (x) \in \Av_{n-1}(123^{\star})\}| = f_n + g_n$.
\end{proof}

The main result of this section, \Cref{main}, follows immediately.

\begin{proof}[Proof of \Cref{main}]
    By Theorems \ref{char} and \ref{bijection}, $|\Sort_{n}(132, 321)| = |\Av_{n}(123, 132^{\star})| = |\Av_{n}(132, 123^{\star})|$. Now, let $G(z)=\sum_{n\ge 0}g_nz^n$. By Lemmas \ref{recursion1} and \ref{recursion2},
    $$
    g_n+g_{n-1}-g_{n-2}=\sum_{i=0}^{n-1}g_ig_{n-1-i},
    $$
    from which
    \begin{align*}
              G(z)-z-1+z(G(z)-1)-z^2G(z) & = zG(z)^2-z \\
        \iff\ zG(z)^2+(z^2-z-1)G(z)+z+1 & =0.
    \end{align*}
    Thus,
    \begin{equation*}
        \frac{(1+z-z^2) - \sqrt{1-2z-5z^2-2z^3+z^4}}{2z}
    \end{equation*}
    is the generating function for $g_n = |\Av_{n}(132, 123^{\star})| = |\Sort_{n}(132, 321)|$. The statement of the theorem now follows from OEIS A102407~\cite{oeis}. 
\end{proof}

\subsection{A bijection with pattern-avoiding Dyck paths} \label{Dyck}

The OEIS A102407~\cite{oeis} counts the set of Dyck paths of semilength $n$ that do not contain the word $dudu$ as a factor. We give an alternative proof of \Cref{main} by showing that $\Av_n(123,132^{\star})$ is in bijection with the set of Dyck paths of semilength $n$ that do not contain the word $dudu$ as a factor. We sketch the construction below and leave the details to the reader.

A \emph{Dyck path} is a path in the discrete plane $\mathbb{Z}\times \mathbb{Z}$ starting at the origin, ending on the $x$-axis, never falling below the $x$-axis and using two kinds of steps---up steps $u=(1,1)$ and down steps $d=(1,-1)$. The \emph{semilength} of a Dyck path is the total number of its up steps, which is also the number of its down steps. 

The \emph{ltr-min decomposition} of a permutation $x$ with ltr minima $m_1,m_2,\dots,m_k$ is
$$
x=m_1B_1m_2B_2\cdots m_kB_k,
$$
where $B_j$ contains the entries between two ltr minima $m_j$ and $m_{j+1}$ for $1 \le j\le k-1$, and $B_k$ contains the entries to the right of $m_k$. The ltr-minima decomposition of $x$ induces a \emph{grid decomposition}~\cite{CERBAI2020P3.32} of $x$, in which
$$
H_i=\left\lbrace y \in \{ 1,2,\ldots ,n\} :m_i < y < m_{i-1} \right\rbrace
$$
is the $i$-th \emph{horizontal strip} of $x$ (note, $m_0=+ \infty$), and $B_j$ is the $j$-th \emph{vertical strip} of $x$. For any two indices $i,j$, the \emph{cell} of indices $i,j$ of $x$ is $C_{i,j}=H_i\cap B_j$. It is clear that $C_{i,j} = \emptyset$ if $i>j$. The grid decomposition of $x=8 \,11 \,6 \,10 \,4 \,9 \,7 \,5 \,3 \,1 \,2$ is depicted in \Cref{grid}. We first prove that each cell of a $123$-avoiding permutation $x$ contains at most one number if and only if $x$ is also $132^{\star}$-avoiding.
  
\begin{proposition} \label{ltrmin dec}
If $x\in\Av_n(123)$, then $x \in \Av_{n}(132^{\star})$ if and only if $|C_{i,j}| \leq 1$ for all $i, j$.
\end{proposition}
\begin{proof}
Suppose that $|C_{i,j}| \geq 2$. We prove that $x$ contains $132^{\star}$. First, the numbers in $B_j$ must all be greater than $m_j$ and appear in decreasing order in $x$, because $x \in \Av_{n}(123)$. Therefore, the two leftmost entries of $x$ that are in $C_{i, j}$, say $x_{\ell}$ and $x_{\ell+1}$, must be adjacent in $x$. Similarly, once again since $x\in\Av(123)$, the numbers in $H_i$ must appear in decreasing order, and thus, $x_{\ell}=x_{\ell+1}+1$. As a result, $m_ix_\ell x_{\ell+1}$ is an occurrence of $132^{\star}$ in $x$.

%Now, $j \geq i$, because $C_{i, j} = \emptyset$ if $i > j$.
%Since $x$ avoids $123$, the numbers in $H_i$ that appear to the right of $x_i$ must appear in decreasing order from the left to right in $x$. Similarly, because $x$ avoids $123$, it must be that $x_\ell>x_{\ell+1}$, because otherwise, $m_ix_\ell x_{\ell+1}$ would be an occurrence of $123$. Thus, $m_ix_\ell x_{\ell+1}$ is an occurrence of $132^{\star}$ in $x$, and so $|C_{i, j}| \leq 1$.

Conversely, if $x_k x_{\ell}x_{\ell+1}$ is an occurrence of $132^{\star}$, then neither $x_\ell$ nor $x_{\ell+1}$ can be an ltr minimum of $x$, and thus, $x_{\ell}$ and $x_{\ell+1}$ are in the same cell.
\end{proof}

As a result, the grid decomposition of a $(123,132^{\star})$-avoiding permutation satisfies the following two properties:
\begin{itemize}
\item Each cell $C_{i,j}$ contains at most one point;
\item If a cell $C_{i,j}$ is not empty, then all the cells located strictly northeast of $C_{i,j}$ (i.e. all the cells $C_{u,v}$, with $u<i$ and $v>j$) are empty.
\end{itemize}
%Conversely, it is easy to see that any permutation whose grid decomposition satisfies the above two properties determines uniquely a $(123,132^{\star})$-avoiding permutation. Indeed, points in the same horizontal strip, as well as points in the same vertical strip, must appear in decreasing order, and this is enough to reconstruct the permutation uniquely.

\begin{figure}
\centering
\begin{tikzpicture}[scale=1, baseline=20.5pt]
      \fill[NE-lines] (3,2) rectangle (4,3);
      \node[scale=1] at (3.5,2.5) {$C_{3,4}$};
      \draw [semithick] (4,0) grid (5,5);
      \draw [semithick] (3,1) grid (4,5);
      \draw [semithick] (2,2) grid (3,5);
      \draw [semithick] (1,3) grid (2,5);
      \draw [semithick] (0,4) grid (1,5);
      \filldraw (0,4) circle (2pt);
      \filldraw (0.5,4.75) circle (2pt);
      \filldraw (1,3) circle (2pt);
      \filldraw (1.5,4.5) circle (2pt);
      \filldraw (2,2) circle (2pt);
      \filldraw (2.25,4.25) circle (2pt);
      \filldraw (2.5,3.5) circle (2pt);
      \filldraw (2.75,2.5) circle (2pt);
      \filldraw (3,1) circle (2pt);
      \filldraw (4,0) circle (2pt);
      \filldraw (4.5,0.5) circle (2pt);
      \node[scale=1] at (0.5,5.5) {$B_1$};
      \node[scale=1] at (1.5,5.5) {$B_2$};
      \node[scale=1] at (2.5,5.5) {$B_3$};
      \node[scale=1] at (3.5,5.5) {$B_4$};
      \node[scale=1] at (4.5,5.5) {$B_5$};
      \node[scale=1] at (5.5,4.5) {$H_1$};
      \node[scale=1] at (5.5,3.5) {$H_2$};
      \node[scale=1] at (5.5,2.5) {$H_3$};
      \node[scale=1] at (5.5,1.5) {$H_4$};
      \node[scale=1] at (5.5,0.5) {$H_5$};
      \node[below] at (0,4){$m_1$};
      \node[below] at (1,3){$m_2$};
      \node[below] at (2,2){$m_3$};
      \node[below] at (3,1){$m_4$};
      \node[below] at (4,0){$m_5$};
\end{tikzpicture}
\hspace{2cm}
\begin{tikzpicture}[scale=0.5, baseline=20.5pt]
	\draw [semithick] (0,11)--(1,11)--(1,10)--(3,10)--(3,9)--(5,9)--(5,8)--(6,8)--(6,6)--(7,6)--(7,4)--(10,4)--(10,1)--(11,1)--(11,0);
	\draw [dotted] (0,11)--(11,0);
	\node[scale=1] at (-1,11) {11};
	\node[scale=1] at (-1,10) {10};
	\node[scale=1] at (-1,9) {9};
	\node[scale=1] at (-1,8) {8};
	\node[scale=1] at (-1,7) {7};
	\node[scale=1] at (-1,6) {6};
	\node[scale=1] at (-1,5) {5};
	\node[scale=1] at (-1,4) {4};
	\node[scale=1] at (-1,3) {3};
	\node[scale=1] at (-1,2) {2};
	\node[scale=1] at (-1,1) {1};
	\node[scale=1] at (-1,0) {0};
	\filldraw (0,11) circle (3pt);
	\filldraw (1,11) circle (3pt);
	\filldraw (1,10) circle (3pt);
	\filldraw (2,10) circle (3pt);
	\filldraw (3,10) circle (3pt);
	\filldraw (3,9) circle (3pt);
	\filldraw (4,9) circle (3pt);
	\filldraw (5,9) circle (3pt);
	\filldraw (5,8) circle (3pt);
	\filldraw (6,8) circle (3pt);
	\filldraw (6,7) circle (3pt);
	\filldraw (6,6) circle (3pt);
	\filldraw (7,6) circle (3pt);
	\filldraw (7,5) circle (3pt);
	\filldraw (7,4) circle (3pt);
	\filldraw (8,4) circle (3pt);
	\filldraw (9,4) circle (3pt);
	\filldraw (10,4) circle (3pt);
	\filldraw (10,3) circle (3pt);
	\filldraw (10,2) circle (3pt);
	\filldraw (10,1) circle (3pt);
	\filldraw (11,1) circle (3pt);
	\filldraw (11,0) circle (3pt);
\end{tikzpicture}
\caption{On the left, the grid decomposition of $x=8 \,11 \,6 \,10 \,4 \,9 \,7 \,5 \,3 \,1 \,2$. On the right, the corresponding Dyck path (rotated clockwise).}\label{grid}
\end{figure}
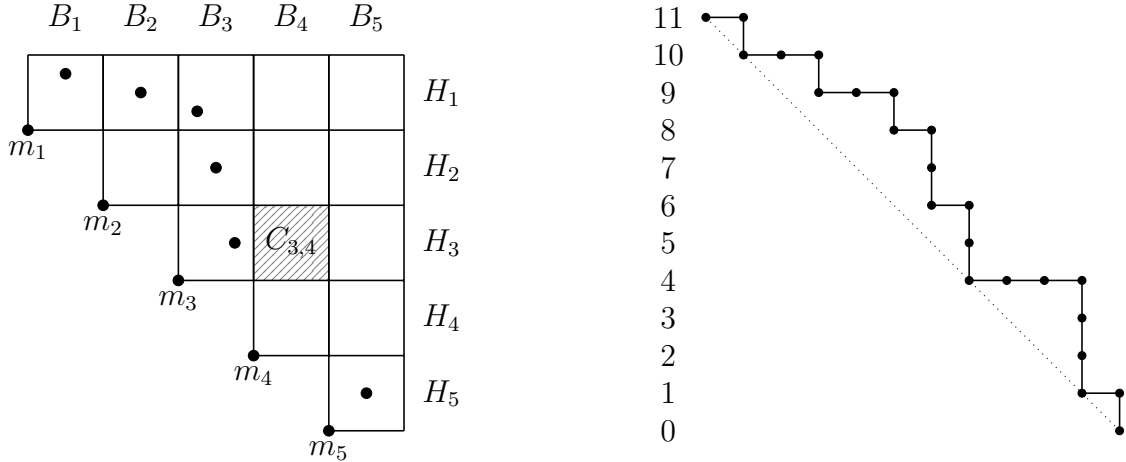

Rotem~\cite{CLKI2008,ROTEM1975} bijected $321$-avoiding permutations with Dyck paths. The Rotem bijection can be reformulated for $123$-avoiding permutations by composing the original map with the complement as follows. Let $x\in\Av_n(123)$. Then construct a sequence $b=b_1b_2\dots b_n$ by letting $b_1=n$, and for $i\ge 2$,
$$
b_i=\begin{cases}
b_{i-1} & \text{if $x_i$ is a ltr minimum of $x$;}\\
x_i-1 & \text{otherwise}.
\end{cases}
$$
For example, the sequence arising from the permutation of \Cref{grid}
$$
x=8 \cdot 11\cdot 6\cdot10\cdot4\cdot9\cdot7\cdot5\cdot3\cdot1\cdot2
\quad\text{is}\quad
b=11\cdot10\cdot10\cdot9\cdot9\cdot8\cdot6\cdot4\cdot4\cdot4\cdot1.
$$
Then represent the sequence $b$ as a bar diagram by drawing a horizontal step of length one starting at $(i-1,b_i)$, for $i=1,\dots,n$, and then connecting the ending point of each step with the starting point of the following one. The bar diagram of $b$ is illustrated in \Cref{grid}. Finally, rotate the bar diagram counterclockwise to arrive at a Dyck path. For example, the Dyck path associated with $b$ in \Cref{grid} is $udu^2du^2dud^2ud^2u^3d^3ud$. 

Now, let $x\in\Av_n(123)$, and let $P$ be the Dyck path associated with $x$ via the reformulated Rotem bijection. Note that $P$ contains $dudu$ as a factor if and only if, in the grid decomposition of $x$, the points corresponding to the up steps in the occurrence of $dudu$ lie in the same cell; instead of giving a formal proof of this fact, we refer the reader to \Cref{dudu}. By \Cref{ltrmin dec}, this is in turn equivalent an occurrence of $132^{\star}$ in $x$. In other words, Rotem bijection maps $(123,132^{\star})$-avoiding permutations to Dyck paths that do not contain $dudu$ as a factor; the desired bijection between $\Sort_n(132,321)$ and Dyck paths avoiding $dudu$ now follows immediately from \Cref{char}.

\begin{figure}
\begin{tikzpicture}[scale=1, baseline=20.5pt]
	\draw [semithick] (0,2)--(0,1)--(1,1)--(1,0)--(2,0);
%	\draw [dotted] (0,11)--(11,0);
%	\node[scale=1] at (-1,11) {11};
	\filldraw (0,2) circle (3pt);
	\filldraw (0,1) circle (3pt);
	\filldraw (1,1) circle (3pt);
	\filldraw (1,0) circle (3pt);
	\filldraw (2,0) circle (3pt);
	\node[above] at (0.5,1){$\pi_{\ell}$};
	\node[above] at (1.5,0){$\pi_{\ell+1}$};
	\node at (2.75,1){$\iff$};
	\node[right] at (3.5,1){$\begin{cases}
	\pi_{\ell}=\pi_{\ell+1}+1;\\
	\pi_{\ell}\text{ is not a ltr minimum}.
	\end{cases}$};
        \node at (9,1){$\iff$};
        \node[above] at (10,0.5) {$m_i$};
        \filldraw (10,0.5) circle (3pt);
        \draw[dotted] (10,0.5) -- (11.5,0.5);
        \draw[semithick] (11.5,0.5)--(13,0.5)--(13,2)--(11.5,2)--(11.5,0.5);
        \filldraw (12,1.5) circle (3pt);
        \filldraw (12.5,0.75) circle (3pt);
        \node[above] at (12,1.5) {$\pi_{\ell}$};
        \node[above]  at (12.5,0.75) {$\pi_{\ell+1}$};
\end{tikzpicture}
\caption{An occurrence of $dudu$ in (the bar diagram of) a Dyck path corresponds to an occurrence of $132^{\star}$ in the associated permutation.}\label{dudu}
\end{figure}
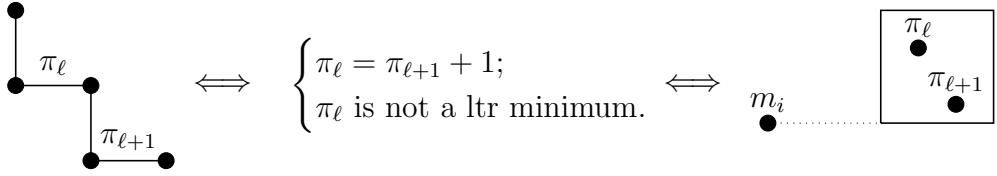

%\section{Proof of \Cref{secondary}}
\section{The $(123,321)$-machine} \label{section_123_321}

In this section, we enumerate the set $\Sort_n(123,321)$ of permutations sorted by the $(123,321)$-machine. A direct computation shows that the sequence $|\Sort_n(123,321)|$, $n\ge 1$, starts with $1,2,4,7$. Our goal is to prove that, for $n\ge 5$, this sequence obeys the simple recursion
$$
|\Sort_n(123,321)|=2|\Sort_{n-1}(123,321)|.
$$
We start by proving that $(123,321)$-sortable permutations avoid $123$.

\begin{proposition} \label{123avoid}
If $x \in \Sort_n (123,321),$ then $x \in \Av_n (123)$.
\end{proposition}
\begin{proof}
Suppose otherwise. Let $x_i x_j x_k$ be the lexicographically smallest occurrence of $123$ in $x$. Note that $x_i$ is a ltr minimum. Now, if $x_i$ is in the stack when $x_k$ enters the stack, $x_j x_k x_i$ is a $231$ pattern in $\mathrm{s}_{123, 321}(x)$, because $x_k x_j x_i$ is a $321$ pattern. Thus, by \Cref{knuth}, $x_i$ must leave the stack before $x_k$ enters the stack. 

Now, because $x_i$ is a ltr minimum, the stack below $x_i$ must be decreasing, since the stack is $123$-avoiding. But the stack must also be $321$-avoiding, so there can be at most two numbers below $x_i$ in the stack. Therefore, for $x_i$ to leave the stack before $x_k$ enters the stack, either $\min \{x_{[i+1, k]}\} < x_{i}$, or the stack below $x_i$ reads $x_{b}x_{c}$ and $\max \{x_{[i+1, k]}\} > x_{b}$. In the first case, $x_i  x_{k} \cdot \min \{x_{[i+1, k]}\}$ is a $231$ pattern in $s_{123, 321}(x)$, and in the second case, $x_{b} \cdot \max \{x_{[i+1, k]}\} \cdot x_c$ is a $231$ pattern in $s_{123, 321}(x)$. Hence, by \Cref{knuth}, $x \in \Av_n (123).$
\end{proof}

In analogy with what observed after \Cref{avoid_123}, we have
$$
\Sort_n(132,321)=\Sort_n(132)\cap\Av(123).
$$
The following two results show that if $x\in\Sort_n(123,321)$, then $x_1 \ge n-1$ and $x_n \leq 2$.

\begin{lemma} \label{begin}
If $x \in \Sort_n (123,321),$ then $x_1 \in \{n-1, n\}$.
\end{lemma}
\begin{proof}
Suppose otherwise. By \Cref{123avoid}, either $x_2 = n$ or $x_2 < x_1$. If $x_2=n$, then $s_{123,321}(x)_{[n-1,n]}=n\cdot x_1$. Therefore, $(n-1) \cdot n \cdot x_1$ is a $231$ pattern in $s_{123,321}(x)$, and $x\not\in \Sort_{n}(123, 321)$ by \Cref{knuth}. Thus, $x_2 < x_1$, and as a result, the stack must only contain $x_1$ when $1$ enters the stack. Then because $x_1$ and $1$ do not exit the stack until the input permutation is empty, $s_{123,321}(x)_{[n-1, n]}= 1 \cdot x_1$. Furthermore, $s_{123,321}(x)_1 = n$, because otherwise, $s_{123,321}(x)_1 \cdot n \cdot 1$ is a $231$ pattern in $s_{123,321}(x)$, and $x \not \in \Sort_{n}(123, 321)$ by \Cref{knuth}. Now, immediately after $n$ leaves the stack, $x_{\indx_{x}(n)-1}$ must leave the stack, because $n$ cannot be part of a $321$ or $123$ pattern after the next number in the input permutation enters the stack. By \Cref{123avoid}, $x_{\indx_{x}(n)-1} \ne n-1$. Thus, $x_{\indx_x(n)-1} \cdot (n-1) \cdot 1$ is a $231$ pattern in $s_{123,321}(x)$---a contradiction by \Cref{knuth}.
\end{proof}

\begin{lemma} \label{end}
If $x \in \Sort_n (123,321)$, then $x_n \in \{1, 2\}$.
\end{lemma}
\begin{proof}
Suppose otherwise. By \Cref{begin}, $x_1 \in \{n-1, n\}$, and so either $x_2 < x_1$, or $x_{[1,2]} = (n-1) \cdot n$. If $x_2 < x_1$, then the stack must only contain $x_1$ when $1$ enters the stack. Otherwise, the stack reads $n \cdot (n-1)$ when $1$ enters the stack. In either case, $1$ cannot leave the stack until the input permutation is empty, and so $1$ appears to the right of both $2$ and $x_n$ in $s_{123,321}(x)$. Now, $2$ must leave the stack before $x_n$ enters the stack, because $x_n\cdot 2 \cdot 1$ is a $321$ pattern. Thus, $2\cdot x_n \cdot1$ is a $231$ pattern in $s_{123,321}(x)$---a contradiction by \Cref{knuth}. 
\end{proof}

The positions of $1$ and $2$ play an important role in our characterization of $(123,321)$-sortable permutations. From now on, assume $n\ge 5$. First we show that $n$ appears to the left of $1$ and $2$ in any $x\in\Sort_n(123,321)$. A consequence is that $2\cdot 1$ must appear as a substring in the output $s_{123,321}(x)$ of the $(123,321)$-stack.

\begin{lemma} \label{after}
If $x \in \Sort_n (123,321)$, then $\indx_x (n) < \min (\indx_x (1), \indx_x (2))$. 
\end{lemma}
\begin{proof}
By \Cref{begin}, $x_1 \geq n-1$. If $x_1=n$, then the claim is immediate. Thus, suppose that $x_1=n-1$. By \Cref{end}, $x_n \in  \{1,2\}$. Now, for the sake of contradiction, suppose that $\indx_x(n) > \indx_x(3-x_n)$. Then the stack must only contain $n-1$ when $3-x_n$ enters the stack and must read $(3-x_n) \cdot (n-1)$ when $n$ enters the stack. Then $n \cdot (3-x_n) \cdot (n-1)$ cannot leave the stack until the input permutation is empty. Thus, $3 \cdot n \cdot (3-x_n)$ is a $231$ pattern in $s_{123,321}(x)$---a contradiction by \Cref{knuth}.
\end{proof}

\begin{lemma}\label{one_two}
    If $x \in \Sort_{n}(123, 321)$, then $\indx_{s_{123,321}(x)}(2) = \indx_{s_{123,321}(x)}(1)-1$. 
\end{lemma}
\begin{proof}
    By Lemmas \ref{begin} and \ref{end}, $x_{1} \in \{n-1, n\}$, and $x_{n} \in \{1, 2\}$. Therefore, just after $3-x_{n}$ enters the stack, the stack reads $(3-x_{n}) \cdot x_1$ or $(3-x_{n}) \cdot n \cdot (n-1)$, and $3-x_{n}$ must remain in the stack until $x_{[1, n-1]}$ has entered the stack. Furthermore, from \Cref{after}, $n$ cannot be above $3-x_n$ in the stack. Now, if $x_{n}= 1$, then $3-x_n=2$ must be the last number that leaves the stack before $1$ enters the stack; after which, $x_n =1$ must leave the stack, because the input permutation is empty. If $x_{n} =2 $, then $x_n=2$ must be right above $3-x_n = 1$ when $x_n =2$ enters the stack; after which, $x_n =2$ and $3-x_n =1$ must exit the stack in that order, because the input permutation is empty.  
\end{proof}

For the rest of this section, given a permutation $x\in S_n$ of length two or more, define $\mathrm{swap}(x)$ as the permutations obtained from $x$ by interchanging $1$ and $2$. %In other words, $\mathrm{swap}(x)$ satisfies $\indx_{\mathrm{swap}(x)} (1) = \indx_{x} (2)$ and $\indx_{\mathrm{swap}(x)} (2) = \indx_{x} (1)$, but $\indx_x (c) = \indx_{\mathrm{swap}(x)} (c)$ for all $c \not\in \{1, 2\}$. For example, $\mathrm{swap}(15324) = 25314$.
Furthermore, for $1\le i\le n+1$, let $\mathrm{inc}_{i}(x)$ be the permutation obtained by inserting $1$ in the $i$-th site of $x$, and suitably rescaling the other entries. More formally, let
$$
\mathrm{inc}_{i}(x)=x'_1\dots x'_{i-1} 1 x'_{i}\dots x'_n,
$$
where $x'_j=x_j+1$. For example, $\mathrm{inc}_{2}(12543) = 213654$.
It follows from the proof of \Cref{one_two} that if $n\geq 5$, then $\mathrm{swap}$ preserves $s_{123, 321}(x)$ for $x \in \Sort_{n}(123, 321)$.  

\begin{corollary} \label{swap}
If $x \in \Sort_n (123,321)$, then $s_{123,321}(x) = s_{123,321} (\mathrm{swap}(x))$.
\end{corollary}

To complete our description of $\Sort_n(123,321)$, we shall prove that a permutation $x$ is $(123,321)$-sortable if and only if $\mathrm{inc}_{n+1}(x)$ is. 

\begin{proposition} \label{extend}
It holds that $x \in \Sort_n (123,321)$ if and only if $\mathrm{inc}_{n+1} (x) \in \Sort_{n+1} (123,321)$.
\end{proposition}
\begin{proof}
First, suppose that $x \in \Sort_{n}(123, 321)$. Then by \Cref{knuth}, $s_{123, 321}(x) \in \Av_{n}(231)$. Now, by \Cref{one_two}, $s_{123, 321}(\mathrm{inc}_{n+1}(x)) = \mathrm{inc}_{\indx_{x}(1)+1}(s_{123, 321}(x))$. Thus, $s_{123, 321}(\mathrm{inc}_{n+1}(x)) \in \Av_{n+1}(231)$, and so $s_{123, 321}(\mathrm{inc}_{n+1}(x)) \in \Sort_{n+1}(123, 321)$ by \Cref{knuth}.

Now, suppose that $x \not \in \Sort_{n}(123, 321)$. Then by \Cref{knuth}, $s_{123, 321}(x) \not\in \Av_{n}(231)$. Now, $s_{123, 321}(\mathrm{inc}_{n+1}(x)) \not\in \Av_{n}(231)$, for $s_{123, 321}(x)+1$ is a substring of $s_{123, 321}(\mathrm{inc}_{n+1}(x))$. Thus, $s_{123, 321}(\mathrm{inc}_{n+1}(x)) \not \in \Sort_{n+1}(123, 321)$ by \Cref{knuth}.
\end{proof}

 Finally, we use Lemmas \ref{end}, \ref{swap}, and \ref{extend} to prove \Cref{secondary}.

\begin{proof}[Proof of \Cref{secondary}]
It is easy to check that the theorem statement holds for $n=4$. Now, by \Cref{swap}, $x \rightarrow \mathrm{swap}(x)$ bijects $\Sort_{n} (123, 321)$ to itself for $n \geq 5$. Thus, $|\Sort_{n} (123,321)| = 2|\{x \in \Sort_{n} (123,321) \mid x_n = 1\}|$ by \Cref{end} for $n \geq 5$. Also, by \Cref{extend}, $x \rightarrow \mathrm{inc}_{n+1} (x)$ bijects $\Sort_n (123,321)$ to $\{x \in \Sort_{n+1} (123,321) \mid x_n=1\}$ for $n \geq 5$. Thus, $|\Sort_{n}(123, 321)| = 2|\Sort_{n-1}(123, 321)|$ for $n \geq 5$ as sought. 
\end{proof}

\section{Future Directions}\label{future}

In addition to the statement of \Cref{main}, Baril, Cerbai, Khalil, and Vajnovski~\cite{BARIL2021106138} conjectured that $|\Sort_{n}(132,213)| = |\Sort_{n}(213,312)|$ for all $n \geq 1$; the conjecture remains open. Here, we refine their conjecture. 

\begin{conjecture} 
   For any $n$ and $1 \leq i \leq n$, we have the equidistributions
\begin{align*}
|\{x \in \Sort_{n}(132, 213) \mid x_1 = i \}| &= |\{x \in \Sort_n(213, 312) \mid x_1 = i\}|
\shortintertext{and}
|\{x \in \Sort_{n}(132, 213) \mid x_i = n \}| &= |\{x \in \Sort_n(213, 312) \mid x_i = n\}|.
\end{align*}
\end{conjecture}

\nocite{*}
\printbibliography

\end{document}